\documentclass[12pt]{amsart}
\usepackage{amssymb,latexsym}
\usepackage{graphicx}
\usepackage{amsmath}
\usepackage{amsfonts}
\usepackage{amssymb}
\usepackage{amsthm}
\usepackage{tabularx}
\usepackage{color}
\usepackage{latexsym}
\usepackage{xypic}
\usepackage[T1]{fontenc}
\usepackage{tgbonum}
\usepackage[height=190mm,width=140mm]{geometry}
\catcode`@=11 \@addtoreset{equation}{section} \catcode`@=12
\newtheorem{theorem}{Theorem}[section]
\newtheorem{corollary}[theorem]{Corollary}

\newtheorem{lemma}[theorem]{Lemma}

\theoremstyle{remark}
\newtheorem{definition}[theorem]{Definition}

\newcommand{\p}{\mathcal{P}}

\begin{document}

\title[A New Operator of Primal Topological  Spaces]
{A New Operator of Primal Topological  Spaces}

\author{Ahmad Al-Omari, Santanu Acharjee and Murad \"{O}zko\c{c} }
\address{Al al-Bayt University, Faculty of Sciences, Department of Mathematics
P.O. Box $130095$, Mafraq $25113$, Jordan}
\email{omarimutah1@yahoo.com}

\address{ Department of Mathematics, Guwahati-$781014$, Assam, India}
\email{sacharjee326@gmail.com}

\address{ Mu\u{g}la S{\i}tk{\i} Ko\c{c}man University,
Faculty of Science, Department of Mathematics  $48000$,
Mente\c{s}e-Mu\u{g}la, Turkey}

\email{murad.ozkoc@mu.edu.tr}

 {
}

\maketitle
\begin{abstract}

Recently, Acharjee et al. [S. Acharjee, M. Özkoç and F. Y. Issaka, \textit{Primal topological spaces}, arXiv:2209.12676[math.GM]] introduced a new structure in topology named primal. Primal is the duel structure of grill. The main purpose of this paper is to introduce an operator using  primal and obtain some of its fundamental properties. Also, we define the notion of topology suitable for a primal. We not only obtain some characterizations of this new notion, but also investigate many properties.
\end{abstract}
\maketitle

\section{Introduction}
Topology is  one of the indispensable branches of mathematics. Due to its various generalized applications in both science and social science, several associated  structures such as ideal \cite{Ja}, filter \cite{Ku}, grill \cite{Cho}, etc. have been introduced. The notion of ideal is the dual of filter. Applications of filter are highly available in literature \cite{MO, MO1}. Moreover, ideal and its generalized fuzzy notion  have been finding several uses in general topology \cite{Al2}, summability theory \cite{BC1}, fuzzy summability theory \cite{ BC2}, and  many others. Similar to ideal,  one of the classical structures in topology is grill. The definition of grill was introduced by Ch\'oquet \cite{Cho} in 1947. Like ideal and filter, grill also has several applications in general topology \cite{Al1}, fuzzy topology \cite{Az}, and many others. Recently,  the notion of primal was introduced by Acharjee et al  \cite{aoi}. Primal is the dual structure of grill. In \cite{aoi}, authors obtained many fundamental properties of this new structure. On the other hand, the set operator $\Psi$ on ideal topological space was introduced and studied by Hamlett and Jankovi\'c  \cite{Ha}. Recently,  different types of $\Psi$-operator can be found  in \cite{Al1, Al2, Al3, Is}, and many others.

In this paper, we continue to study the properties of the operator $\diamond$ defined in  \cite{aoi}. We define an operator $\Psi$ using primal and investigate its various fundamental properties. Also, we introduce the notion of topology suitable for a primal and obtain its characterizations along with  several properties.\\

\section{Preliminaries}

Throughout this paper, $(X,\tau)$ and $(Y,\sigma)$ (briefly, $X$ and $Y$) represent topological spaces unless otherwise stated. For any subset $A$ of a space $X$, $cl(A)$ and $int(A)$ denote  closure and interior of $A$, respectively. The powerset of a set $X$ will be denoted by $2^X.$ The family of all open neighborhoods of a point $x$ of $X$ is denoted by $\tau(x).$ Also, the family of all closed subsets of a space $X$ will be denoted by $C(X).$ Now, we procure following notions and results which will be required in next section.

\begin{definition} $($\cite{Cho}$)$
A family $\mathcal{G}$ of $2^X$ is called a grill  on $X$ if $\mathcal{G}$ satisfies the following conditions:
\begin{enumerate}
  \item $\emptyset\notin \mathcal{G},$
  \item if $A\in\mathcal{G}$ and $A\subseteq B,$ then $B\in \mathcal{G}$,
  \item if $A\cup B\in \mathcal{G},$ then $A\in\mathcal{G}$ or $B\in\mathcal{G}.$
\end{enumerate}

\end{definition}

\begin{definition} $($\cite{aoi}$)$
Let $X$ be a nonempty set. A collection $\mathcal{P}\subseteq   2^X$  is called a primal on $X$ if it satisfies the following conditions:
\begin{enumerate}
  \item $X \notin \mathcal{P}$,
  \item if $A\in \mathcal{P}$ and $B\subseteq A$, then $B\in \mathcal{P}$,
  \item if $A\cap B\in \mathcal{P}$, then $A\in \mathcal{P}$ or $B\in \mathcal{P}$.
\end{enumerate}
\end{definition}

\begin{corollary} $($\cite{aoi}$)$ \label{23}
Let $X$ be a nonempty set. A collection $\mathcal{P}\subseteq   2^X$  is a primal on $X$ if and only if it satisfies the following conditions:
\begin{enumerate}
  \item $X \notin \mathcal{P}$,
  \item if $B\notin \mathcal{P}$ and $B\subseteq A$, then $A\notin \mathcal{P}$,
  \item if $A\notin \mathcal{P}$ and $B\notin \mathcal{P},$ then $A\cap B\notin \mathcal{P}.$
\end{enumerate}
\end{corollary}

\begin{definition}$($\cite{aoi}$)$
A topological space $(X,\tau)$ with a primal $\mathcal{P}$ on $X$ is called a primal topological space and denoted by $(X,\tau,\mathcal{P}).$
\end{definition}

\begin{definition} (\cite{aoi})
Let $(X,\tau,\mathcal{P})$ be a primal topological space. We consider a map $(\cdot)^{\diamond}: 2^X \rightarrow 2^X$ as $A^{\diamond}(X,\tau,\mathcal{P})=\{x\in X : (\forall U\in \tau (x))(A^c\cup U^c \in \mathcal{P})\}$ for any subset $A$ of $X.$ 
We can also write $A_{\mathcal{P}}^{\diamond}$ as $A^{\diamond}(X,\tau,\mathcal{P})$ to specify the primal as per our requirements.  
\end{definition}

\begin{definition}$($\cite{aoi}$)$
Let $(X,\tau,\mathcal{P})$ be a primal topological space. We consider a map $cl^{\diamond}: 2^X\rightarrow 2^X$ as $cl^{\diamond}(A)=A\cup A^{\diamond}$, where $A$ is any subset of $X$. 
\end{definition}

\begin{definition}$($\cite{aoi}$)$
Let $(X,\tau,\mathcal{P})$ be a primal topological space. Then, the family $\tau^{\diamond}=\{A\subseteq X|cl^{\diamond}(A^c)=A^c\}$ is a topology on $X$ induced by topology $\tau$ and primal $\mathcal{P}.$  It is called primal topology on $X.$ We can also write $\tau_{\mathcal{P}}^{\diamond}$ instead of $\tau^{\diamond}$ to specify the primal as per our requirements.
\end{definition}

\begin{theorem}$($\cite{aoi}$)$ ~\label{12}
Let $(X, \tau, \p)$  be a primal topological  space and $A, B\subseteq X$. If $A$ is open in $X$, then $A\cap  B^{\diamond}\subseteq (A\cap B)^{\diamond}$.
\end{theorem}

\begin{corollary}~\label{14}
Let $(X, \tau, \p)$  be a primal topological  space and $A, B\subseteq X$. If $A$ is open in $X$, then $A\cap  B^{\diamond}=A\cap(A\cap B)^{\diamond} \subseteq (A\cap B)^{\diamond}$.
\end{corollary}
\begin{proof} We have $A\cap(A\cap B)^{\diamond}\subseteq A\cap  B^{\diamond}$. Let $x\in A\cap  B^{\diamond}$ and $V\in \tau(x)$. Then $A\cap V\in \tau(x)$ and $x\in B^{\diamond}$, then $(A\cap V)^{c}\cup B^{c}\in \mathcal{P}$ i.e. $(A\cap B)^{c}\cup V^{c}\in \mathcal{P}$, then $x\in (A\cap B)^{\diamond}$ and $x\in A\cap(A\cap B)^{\diamond}$. Hence by Theorem ~\ref{12}, we have $A\cap  B^{\diamond}=A\cap(A\cap B)^{\diamond} \subseteq (A\cap B)^{\diamond}$.
\end{proof}

\begin{theorem}$($\cite{aoi}$)$ \label{1a}
Let $(X,\tau,\mathcal{P})$ be a primal topological space. Then, the following statements hold for any two subsets of $A$ and $B$ of $X.$

$(1)$ if $A^c\in \tau,$ then $ A^{\diamond}\subseteq A,$

$(2)$ $\emptyset^{\diamond}= \emptyset$,

$(3)$ $cl(A^{\diamond})=A^{\diamond},$ 

$(4)$ $(A^{\diamond})^{\diamond}\subseteq A^{\diamond},$ 

$(5)$ if $A\subseteq B$, then $A^{\diamond}\subseteq B^{\diamond}$,

$(6)$ $A^{\diamond}\cup B^{\diamond}= (A\cup B)^{\diamond},$


$(7)$ $ (A\cap B)^{\diamond}\subseteq A^{\diamond}\cap B^{\diamond}.$
\end{theorem}

\begin{theorem}$($\cite{aoi}$)$ \label{base}
Let $(X,\tau,\mathcal{P})$ be a primal topological space. Then, the family $\mathcal{B}_{\mathcal{P}}=\{T\cap P|T\in\tau$ and $ P\notin \mathcal{P}\}$ is a base for the primal topology $\tau^{\diamond}$ on $X.$
\end{theorem}

\section{Main Results}

In this section, we introduce some results related to $\Psi$ operator using primal on a primal topological space $(X, \tau, \p)$.\\ 
\begin{theorem}~\label{7}
Let $(X, \tau, \p)$  be a primal topological  space. If $C(X)-\{X\}\subseteq \p,$ then $U\subseteq U^{\diamond}$ for all $U\in \tau.$	
\end{theorem}

\begin{proof}
	In case $U =\emptyset$, we obviously have $U^{\diamond}=\emptyset=U$. Now note that
if $C(X)-\{X\}\subseteq \p$, then $X^{\diamond}= X$. In fact $x \notin X^{\diamond},$ then there is  $V\in \tau(x)$ such that
$V^{c} \cup X^{c} \notin\p$. Hence, $V^{c}\notin\p$ is a contradiction.
Now by using Theorem \ref{12}, we have for any $U \in \tau$, $U=U \cap X^{\diamond} \subseteq( U \cap X)^{\diamond}=U^{\diamond}$. Thus $U\subseteq U^{\diamond}$.
\end{proof}
\begin{lemma}~\label{11}
Let $(X, \tau, \p)$  be a primal topological  space.  If $A^{c}\notin  \p$, then $ A^{\diamond}=\emptyset.$
\end{lemma}
\begin{proof} Suppose that $x\in A^{\diamond}$. Then, for any open set $U$ containing $x$ we have  $U^{c}\cup A^{c}\in \p$. Since $A^{c}\notin \p$, $U^{c}\cup A^{c}\notin \p$ for some  open sets $U$ containing $x$. This is a contradiction. Hence, $A^{\diamond}=\emptyset$.
\end{proof}

\begin{lemma}~\label{10}
Let $(X, \tau, \p)$  be a primal topological  space  and $A, B $ be subsets of $ X$. Then, $A^{\diamond}- B^{\diamond}=(A-B)^{\diamond}-B^{\diamond}$.
\end{lemma}
\begin{proof}
We have by Theorem \ref{1a}, $A^{\diamond}=[(A-B)\cup (A\cap B)]^{\diamond}=(A-B)^{\diamond} \cup (A\cap B)^{\diamond} \subseteq (A-B)^{\diamond} \cup B^{\diamond}$. Thus  $A^{\diamond}-B^{\diamond}\subseteq (A-B)^{\diamond}-B^{\diamond}$. Again by Theorem \ref{1a}, $(A-B)^{\diamond}\subseteq A^{\diamond}$ and hence, $(A-B)^{\diamond}-B^{\diamond}\subseteq A^{\diamond}-B^{\diamond}$.  Hence, $A^{\diamond}-B^{\diamond}=(A-B)^{\diamond}-B^{\diamond}$.
\end{proof}

\begin{corollary}~\label{5}
Let  $(X, \tau, \p)$  be a primal topological  space and $A, B $ be subsets of $X$ with $B^{c}\notin \p$. Then  $(A\cup B)^{\diamond}=A^{\diamond}=(A-B)^{\diamond}$.
\end{corollary}

\begin{proof}
Since $B^{c}\notin \p$, thus $B^{\diamond}=\emptyset$. Again by Lemma \ref{10}, $A^{\diamond}=(A-B)^{\diamond}$ and by Theorem \ref{1a},  $(A\cup B)^{\diamond}=A^{\diamond}\cup B^{\diamond}=A^{\diamond}$
\end{proof}



\begin{definition}
Let $(X,\tau, \p)$ be a primal topological  space. An operator $\Psi:2^{X}\rightarrow  2^{X}$ defined  as $\Psi(A)=\{x\in X: (\exists U\in \tau(x))((U-A)^{c} \notin \p) \}$ for every $A\subseteq X.$
\end{definition}
Several basic facts concerning the behavior of the operator $\Psi$ are included in the following theorem.
\vspace{2mm}
\begin{theorem}~\label{3}
Let $(X, \tau, \p)$ be a primal topological  space. Then, the following properties hold:
\begin{enumerate}
\item if $A\subseteq X$, then $\Psi(A)=X-(X-A)^{\diamond}$,
  \item if $A\subseteq X$, then $\Psi(A)$ is open,
  \item if $A\subseteq B$, then $\Psi(A)\subseteq \Psi(B)$,
  \item  if $A, B\subseteq X$, then $\Psi(A\cap B)=\Psi(A)\cap \Psi(B)$,
  \item if $U\in \tau^{\diamond}$, then $U\subseteq \Psi(U)$,
  \item if $A\subseteq X$, then $\Psi(A)\subseteq \Psi(\Psi(A))$,
  \item if $A\subseteq X$, then $\Psi(A)=\Psi(\Psi(A))$ if and only if\\ $(X-A)^{\diamond}=((X-A)^{\diamond})^{\diamond}$,
  \item if $A^{c}\notin \p$, then $\Psi(A)=X-X^{\diamond}$,
  \item if $A\subseteq X $, then $A\cap \Psi(A)=int^{\diamond}(A)$,
  \item if $A\subseteq X $ and $I^{c}\notin \p$, then $\Psi(A-I)= \Psi(A)$,
  \item if $A\subseteq X $ and $I^{c}\notin \p$, then $\Psi(A\cup I)= \Psi(A)$,
  \item if $[(A-B) \cup (B-A)]^{c}\notin \p$, then $\Psi(A)= \Psi(B)$.
\end{enumerate}
\end{theorem}
 \begin{proof}
(1) Let $x\in \Psi(A)$, then there exists $U\in \tau(x)$ such that $U^{c}\cup A=(U\cap (X-A))^{c}=(U-A)^{c}\notin \p$,
 then $x\notin (X-A)^{\diamond}$ and $x\in X-(X-A)^{\diamond}$.
Conversely, let  $x\in X-(X-A)^{\diamond}$, then $x\notin (X-A)^{\diamond}$, then there exists $U\in \tau(x)$ such that $U^{c}\cup (X-A)^{c}=(U-A)^{c}\notin \p$. Hence, $x\in \Psi(A)$ and $\Psi(A)=X-(X-A)^{\diamond}$.
 
(2)  This follows from (3) of Theorem \ref{1a}.

(3)  This follows from (5) of Theorem \ref{1a}.

(4) It follows from (3) that $\Psi(A\cap B)\subseteq\Psi(A)$ and $\Psi(A\cap B)\subseteq \Psi(B)$. Hence, $\Psi(A\cap B)\subseteq\Psi(A)\cap \Psi(B)$. Now, let $x\in \Psi(A)\cap \Psi(B)$. Then, there exist $U, V\in \tau(x)$ such that $(U-A)^{c}\notin \p$ and  $(V-B)^{c}\notin \p$. Let $G=U\cap V \in \tau(x)$ and we have $(G - A)^{c}\notin \p$ and $(G- B)^{c}\notin \p$ by heredity. Thus $[G-(A\cap B)]^{c}=(G-A)^{c}\cap (G-B)^{c}\notin \p$ by Corollary ~\ref{23}, and hence, $x\in \Psi(A\cap B)$. We have shown $\Psi(A)\cap \Psi(B)\subseteq \Psi(A\cap B)$ and the proof is completed.

(5) If $U\in \tau^{\diamond}$, then  $(X-U)^{\diamond}\subseteq X-U.$ Hence, $U\subseteq X-(X-U)^{\diamond}=\Psi(U)$.

(6) It follows from (2) and (5).

(7) It follows from the facts:
\begin{enumerate}
  \item[(a)] $\Psi(A)=X-(X-A)^{\diamond}$.
  \item[(b)] $\Psi(\Psi(A))=X-[X-(X-(X-A)^{\diamond})]^{\diamond}=X-((X-A)^{\diamond})^{\diamond}$.
\end{enumerate}

(8) By Corollary ~\ref{5},  we obtain that $(X-A)^{\diamond}=X^{\diamond}$ if $A^{c}\notin \p$. Then, $\Psi(A)=X-(X-A)^{\diamond}=X-X^{\diamond}.$

(9) If $x\in A\cap \Psi(A)$, then $x\in A$ and there exists $U_{x}\in \tau(x)$ such that $(U_{x}-A)^{c}\notin \p$. Then by Theorem ~\ref{base}, $U_{x}\cap (U_{x}-A)^{c}$ is a $\tau^{\diamond}$-open neighborhood of $x$ and $x\in int^{\diamond}(A)$. On the other hand, if $x\in int^{\diamond}(A)$,
there exists a basic $\tau^{\diamond}$-open neighborhood $V_{x}\cap I$ of $x$, where $V_{x}\in \tau$ and  $I\notin \p$, such that $x\in V_{x}\cap I\subseteq A$  which implies $I\subseteq (V_{x}- A)^{c}$ and hence, $(V_{x}- A)^{c} \notin \p$. Hence, $x\in A\cap \Psi(A)$.

(10) This follows from Corollary ~\ref{5}  and  $\Psi(A-I)=X-[X-(A-I)]^{\diamond}=X-[(X-A)\cup I]^{\diamond}=X-(X-A)^{\diamond}=\Psi(A)$.

(11) This follows from Corollary ~\ref{5}  and  $\Psi(A \cup I)=X-[X-(A\cup I)]^{\diamond}=X-[(X-A)- I]^{\diamond}=X-(X-A)^{\diamond}=\Psi(A)$.

(12) Assume $[(A-B)\cup (B-A)]^{c}\notin \p$. Let $A-B=I$ and $B-A=J$. Observe that $I^{c}, J^{c} \notin \p$ by heredity. Also, observe that $B=(A-I)\cup J$. Thus,
$\Psi(A)=\Psi (A-I)=\Psi[(A-I)\cup J]=\Psi(B)$ by (10) and (11).
\end{proof}
\vspace{2mm}
\begin{corollary}
Let $(X, \tau, \p)$ be a primal topological space. Then, $U\subseteq \Psi(U)$ for every open set $U\in  \tau$.
\end{corollary}
\begin{proof}
We know that $\Psi(U)=X-(X-U)^{\diamond}$. Now $(X-U)^{\diamond}\subseteq cl(X-U)=X-U$,  since $X-U$ is closed. Therefore, $U= X-(X-U)\subseteq X-(X-U)^{\diamond}=\Psi(U)$.
\end{proof}
\vspace{2mm}
\begin{theorem}~\label{13}
Let $(X, \tau, \p)$ be a primal topological  space and $A\subseteq X$. Then, the following properties hold:
 \begin{enumerate}
   \item $\Psi (A)=\bigcup  \{U\in  \tau : (U-A)^{c}\notin \p\}$,
   \item $\Psi(A)\supseteq\bigcup  \{U\in  \tau : (U-A)^{c}\cup (A-U)^{c}\notin \p\}$.
 \end{enumerate}
\end{theorem}
 \begin{proof}
(1) This follows immediately from the definition of $\Psi$-operator.\\
(2) Since $\p$ is heredity, it is obvious that $\bigcup  \{U\in  \tau : (U-A)^{c}\cup (A-U)^{c}\notin \p\} \subseteq \bigcup  \{U\in  \tau : (U-A)^{c}\notin \p\}=\Psi(A)$, for every $A\subseteq X$.
\end{proof}
\vspace{2mm}
\begin{theorem}~\label{9}
Let $(X, \tau, \p)$ be a primal topological  space. If $\sigma=\{A\subseteq X : A\subseteq \Psi (A)\}$. Then,  $\sigma$ is a topology on $X$ and $\sigma=\tau^{\diamond}$.
\end{theorem}
 \begin{proof}

Let $\sigma=\{A\subseteq X : A\subseteq \Psi (A)\}$. First, we show that $\sigma$ is a topology. Observe that $\emptyset\subseteq \Psi(\emptyset)$ and $X\subseteq \Psi(X)=X$, and thus $\emptyset$ and $X \in \sigma$. Now, if $A, B\in \sigma$, then $A\cap B\subseteq \Psi(A)\cap \Psi(B)=\Psi(A\cap B).$ It implies that $A\cap B \in \sigma.$ If $\{A_{\alpha}: \alpha \in \Delta\}\subseteq \sigma$, then $A_{\alpha} \subseteq \Psi( A_{\alpha})\subseteq \Psi(\bigcup_{\alpha\in\Delta} A_{\alpha})$, for every $\alpha\in \Delta$, and hence  $\bigcup A_{\alpha}\subseteq \Psi(\bigcup_{\alpha\in\Delta} A_{\alpha})$. This shows that $\sigma$ is a topology on $X.$

Now, if $U\in \tau^{\diamond}$  and $x\in U$, then by Theorem ~\ref{base}, there exists $V\in \tau(x)$ and $I\notin \p$ such that $x\in V\cap I\subseteq U$. Clearly $I\subseteq (V-U)^{c}$, so that $(V-U)^{c}\notin \p$ by heredity, and hence $x\in \Psi(U)$. Thus, $U\subseteq \Psi(U)$ and $\tau^{\diamond}\subseteq\sigma$. Now, let $A\in \sigma$, then we have  $A\subseteq \Psi(A)$, i.e. $A\subseteq X-(X-A)^{\diamond}$ and $(X-A)^{\diamond}\subseteq X-A$. This shows that  $X-A$ is $\tau^{\diamond}$-closed and hence $A\in \tau^{\diamond}$. Thus, $\sigma\subseteq \tau^{\diamond}$, and hence $\sigma=\tau^{\diamond}$.
\end{proof}
\section{Topology suitable for a primal}

In this section, we introduce topology suitable for a primal on a primal topological space and investigate some properties.\\

\begin{definition}
Let $(X,\tau,\mathcal{P})$ be a primal topological space. Then, $\tau$  is said to be suitable for the primal $\mathcal{P}$ if $A^{c} \cup A^{\diamond}\notin \mathcal{P}$, for all $A \subseteq X.$

\end{definition}

We now give some equivalent descriptions of this definition.


\begin{theorem}~\label{4}
For a primal topological  space  $(X, \tau, \p)$, the following are equivalent:
\begin{enumerate}
  \item $\tau$ is suitable for the primal $\mathcal{P}$,
  \item for any $\tau^{\diamond}$-closed subset $A$ of $X$, $A^{c} \cup A^{\diamond}\notin \mathcal{P}$,
  \item whenever for any $A\subseteq X$ and each $x\in A$  there corresponds some $U_x \in \tau(x)$
with $U_{x}^{c}\cup A^{c} \notin \mathcal{P}$, it follows that $A^{c} \notin\mathcal{P}$,
\item for $A\subseteq X$ and $A\cap A^{\diamond}=\emptyset$, it follows that $A^{c} \notin\mathcal{P}$.
\end{enumerate}
\end{theorem}
 \begin{proof}
 (1) $\Rightarrow$ (2): It is trivial.

  (2) $\Rightarrow$ (3):  Let $A\subseteq X$ and assume that for every $x\in A$ there exists $U \in \tau(x)$
such that $U^{c} \cup A^{c} \notin \mathcal{P}$. Then, $x\notin A^{\diamond}$ so that $A\cap A^{\diamond} = \emptyset$. Since $A\cup A^{\diamond}$ is
$\tau^{\diamond}$-closed, thus by (2) we have $(A\cup A^{\diamond})^{c} \cup  (A\cup A^{\diamond})^\diamond \notin \mathcal{P}$. i.e. $(A\cup A^{\diamond})^{c} \cup (A^{\diamond}\cup (A^{\diamond})^\diamond) \notin \mathcal{P}$ by Theorem ~\ref{1a} (6)  i.e. $(A\cup A^{\diamond})^{c}\cup A^{\diamond} \notin\mathcal{P}$ by Theorem ~\ref{1a} (4)
i.e., $A^{c} \notin   \mathcal{P}$ (as $A \cap A^{\diamond} = \emptyset$).

(3) $\Rightarrow$ (4):  If $A\subseteq X$ and $A \cap A^{\diamond} = \emptyset$, then $A \subseteq X \setminus A^{\diamond}$. Let $x\in A$.
Then $x\notin A^{\diamond}$. So there exists $U\in \tau(x)$ such that $U^{c} \cup A^{c}\notin \mathcal{P}$. Then by (3), $A^{c} \notin\mathcal{P}$.

    (4) $\Rightarrow$ (1):  Let $A\subseteq X$. We  first claim  that $(A\setminus A^{\diamond})\cap (A\setminus A^{\diamond})^{\diamond}=\emptyset$.
    In fact, if $x\in (A\setminus A^{\diamond})\cap (A\setminus A^{\diamond})^{\diamond}$, then $x\in A\setminus A^{\diamond}.$ Thus $x\in A$ and $x\notin A^{\diamond}.$ Then, there exists $U\in \tau(x)$ such that $U^{c}\cup A^{c}\notin \mathcal{P}$. Now, $U^{c}\cup A^{c}\subseteq U^{c}\cup (A\setminus A^{\diamond})^{c}$ by Corollary ~\ref{23} (2),  $U^{c}\cup (A\setminus A^{\diamond})^{c}\notin \mathcal{P}.$ Hence,  $x\notin (A\setminus A^{\diamond})^{\diamond}$, which is a contradiction. Hence, by (4), $(A \setminus A^{\diamond})^{c}= A^{c} \cup A^{\diamond}\notin \mathcal{P}$ and $\tau$ is suitable for the primal $\mathcal{P}$.
 \end{proof}


\begin{theorem}~\label{1}
For a primal topological  space  $(X, \tau, \p)$, the following conditions are equivalent and any of these three  conditions is   necessary for $\tau$ 
to be suitable for the primal $\mathcal{P}$.
\begin{enumerate}
  \item for any $A\subseteq X$, $A\cap A^{\diamond}=\emptyset$, then $A^{\diamond}=\emptyset$,
   \item for any $A\subseteq X$, $(A\setminus A^{\diamond})^{\diamond}=\emptyset$,
    \item for any $A\subseteq X$, $(A\cap A^{\diamond})^{\diamond}=A^{\diamond}$.
\end{enumerate}
\end{theorem}
 \begin{proof}
(1) $\Rightarrow$ (2): It follows by  noting  that $(A\setminus A^{\diamond})\cap (A\setminus A^{\diamond})^{\diamond}=\emptyset$, for all $A\subseteq X$.
  
(2) $\Rightarrow$ (3): Since $A=(A\setminus (A\cap A^{\diamond}))\cup (A\cap A^{\diamond}),$ we have 
$A^{\diamond}=(A\setminus (A\cap A^{\diamond}))^{\diamond}\cup (A\cap A^{\diamond})^{\diamond}=(A\setminus  A^{\diamond})^{\diamond}\cup (A\cap A^{\diamond})^{\diamond}=(A\cap A^{\diamond})^{\diamond}$ by (2).
    
(3) $\Rightarrow$ (1): Let $A\subseteq X$ and  $A\cap A^{\diamond}=\emptyset$. Then by (3), $A^{\diamond}= (A\cap A^{\diamond})^{\diamond}=\emptyset^{\diamond}=\emptyset.$
  \end{proof}

\begin{corollary}~\label{15}
If  $(X,\tau,\mathcal{P})$ be a primal topological space such that $\tau$ is suitable for
$\mathcal{P}$.,
then the operator $\diamond$ is an idempotent operator i.e., $A^{\diamond}=(A^{\diamond})^{\diamond}$ for any $A\subseteq X$.
\end{corollary}
 \begin{proof}
By  Theorem ~\ref{1a} (4), we have $(A^{\diamond})^{\diamond}\subseteq A^{\diamond}$. By Theorem ~\ref{1} and Theorem ~\ref{1a} (5),
we get $A^{\diamond}=(A\cap A^{\diamond})^{\diamond}\subseteq (A^{\diamond})^{\diamond}.$
\end{proof}
 

\begin{theorem}~\label{2}
Let $(X,\tau,\mathcal{P})$ be a primal topological space such that $\tau$ is suitable for
$\mathcal{P}$. Then, a subset $A$  of $X$ is $\tau^{\diamond}$-closed if and only if  it can be expressed as a union of a  set which is
closed in $(X, \tau)$ and a complement  not in $\mathcal{P}$.
\end{theorem}
\begin{proof}Let $A$ be a $\tau^{\diamond}$-closed subset of $X$. Then, $A^{\diamond}\subseteq A$. Now, $A=A^{\diamond}\cup (A\setminus A^{\diamond})$.
Since $\tau$ is suitable for $\mathcal{P}$, then by Theorem ~\ref{4} $(A\setminus A^{\diamond})^{c}\notin \mathcal{P}$ and by Theorem ~\ref{1a} (3),
$A^{\diamond}$ is closed.

Conversely, let $A=F\cup B$, where $F$ is closed and $B^{c}\notin \mathcal{P}$. Then, $A^{\diamond}=(F\cup B)^{\diamond}=F^{\diamond}$  by Corollary ~\ref{5}, and hence by Theorem ~\ref{1a}(3)  $A^{\diamond}=(F\cup B)^{\diamond}=F^{\diamond}=cl(F)=F\subseteq A$. Hence, $A$ is $\tau^{\diamond}$-closed.
\end{proof}

\begin{corollary}~\label{6}
Let the  topology $\tau$ on  a space  $X$ be suitable for a primal $\mathcal{P}$ on $X$.
Then, $\mathcal{B}_{\mathcal{P}}=\{U\cap P: (U\in \tau)(P\notin \mathcal{P})\}$ is a topology on $X$ and hence, $\mathcal{B}_{\mathcal{P}}=\tau^{\diamond}$
\end{corollary}
 \begin{proof}
Let $U\in \tau^{\diamond}$. Then by Theorem ~\ref{2}, $X\setminus U=F\cup B$, where $F$ is closed and $B^{c}\notin \mathcal{P}$. 
Then, $U=X\setminus(F\cup B)=(X\setminus F)\cap (X\setminus B)=V\cap P$, where $V= F^{c}\in \tau$ and $P=B^{c}\notin \mathcal{P}$.
Thus, every $\tau^{\diamond}$-open set is of the form $V\cap P$, where $V\in \tau$ and $P\notin \mathcal{P}$. The rest follows from Theorem 
 ~\ref{base}.
\end{proof}



\begin{theorem}~\label{8}
Let $(X, \tau, \mathcal{P})$ be a primal topological space  and $A$  be any subset
of $X$ such that $A\subseteq A^{\diamond}$. Then, 
$cl(A)=cl^{\diamond}(A)=cl(A^{\diamond})=A^{\diamond}$.
\end{theorem}
 \begin{proof}
Since $\tau^{\diamond}$ is finer than $\tau$, then $cl^{\diamond}(A)\subseteq cl(A)$ for any subset $A$ of $X$. Now $x\notin cl^{\diamond}(A)$, there exist $V\in \tau$ and $B\in \mathcal{P}$ such that $x\in V\cap B$ and $(V\cap B)\cap A=\emptyset$, then $[(V\cap B)\cap A]^{\diamond}=\emptyset$. Thus $[(V\cap A)\setminus B^{c}]^{\diamond}=\emptyset$, hence by Corollary ~\ref{5} we have $(V\cap A)^{\diamond}=\emptyset$. By Theorem ~\ref{12}, we get $V\cap (A)^{\diamond}=\emptyset$ and  $V\cap A=\emptyset$ (as $A\subseteq A^{\diamond}$), then $x\notin cl(A)$. Thus, $cl(A)=cl^{\diamond}(A)$.
Now, by Theorem ~\ref{1a} (3), $A^{\diamond}=cl(A^{\diamond})$. Now, let $x\notin cl(A).$ Then, there exists $U\in \tau(x)$ such that $U\cap A=\emptyset.$ Thus, $(U\cap A)^{c}=U^{c}\cup A^{c}=X\notin \mathcal{P}$. So, $x\notin A^{\diamond}$ and hence $A^\diamond\subseteq cl(A)$. Again as $A^\diamond\subseteq cl(A)$, so
we have $cl(A^\diamond)\subseteq cl(cl(A))=cl(A)$. Also, $A\subseteq A^{\diamond}$, then $cl(A)\subseteq cl( A^{\diamond})$. Thus, $cl(A)=cl(A^{\diamond})=A^{\diamond}$.
\end{proof}

\begin{theorem}
Let $(X, \tau, \mathcal{P})$ be a primal topological space such that $\tau$ is suitable for
$\mathcal{P}$ with $C(X)-\{X\}\subseteq \mathcal{P}$. Let $G$ be $\tau^{\diamond}$-open set such that $G=U\cap A$, where $U\in \tau$
and $A\notin \mathcal{P}.$ Then, 
$cl(G)=cl^{\diamond}(G)=G^{\diamond}=U^{\diamond}=cl(U)=cl^{\diamond}(U)$.
\end{theorem}
 \begin{proof}
 Let $G=U\cap A$, where $U\in \tau$ and $A\notin \mathcal{P}$ (by Corollary ~\ref{6}, every $\tau^{\diamond}$-open set $G$ is of this form).
 Since $C(X)-\{X\}\subseteq \mathcal{P}$, by Theorem ~\ref{7} we have $U\subseteq U^{\diamond}$. Hence, by Theorem ~\ref{8}, we get $U^{\diamond}=cl(U)=cl^{\diamond}(U)$.
 
 Now, let $G$ be $\tau^{\diamond}$-open. We claim that $G\subseteq G^{\diamond}$. In fact, $cl^{\diamond}(X\setminus G)=X\setminus G$, then 
 $(X\setminus G)^{\diamond}=X\setminus G$ and $X^{\diamond}\setminus G^{\diamond}=X\setminus G$ by Lemma ~\ref{10} and  by Theorem ~\ref{7} we have
 $X\setminus G^{\diamond}=X\setminus G$, $G\subseteq G^{\diamond}$. Hence, by Theorem ~\ref{8}, $G^{\diamond}=cl(G)=cl^{\diamond}(G)$.
 
 Again, $G\subseteq U$, so $G^{\diamond}\subseteq U^{\diamond}$ and also $G^{\diamond}= (U\cap A)^{\diamond}=(U- A^{c})^{\diamond}\supseteq U^{\diamond}- (A^{c})^{\diamond}= U^{\diamond}$ by Lemma ~\ref{10} and Lemma ~\ref{11} as $A\notin \mathcal{P}$. Thus, $U^{\diamond}=G^{\diamond}$. Therefore,  we have $cl(G)=cl^{\diamond}(G)=G^{\diamond}=U^{\diamond}=cl(U)=cl^{\diamond}(U)$.
\end{proof}


\begin{theorem}~\label{16}
Let $(X, \tau, \mathcal{P})$ be a primal topological space such that $\tau$ is suitable for
$\mathcal{P}$. Then, for every $A\in \tau$ and any $B\subseteq X$, 
$(A\cap B)^{\diamond}=(A\cap B^{\diamond})^{\diamond}=cl(A\cap B^{\diamond})$.
\end{theorem}
\begin{proof}
Let $A\in \tau$. Then by Corollary ~\ref{14}, $A\cap  B^{\diamond}=A\cap(A\cap B)^{\diamond} \subseteq (A\cap B)^{\diamond}$
and  hence, $(A\cap  B^{\diamond})^{\diamond} \subseteq [(A\cap B)^{\diamond}]^{\diamond}=(A\cap B)^{\diamond}$ by Corollary ~\ref{15}.

Now by using Corollary ~\ref{14} and Theorem  ~\ref{1},  we obtain $[A\cap  (B\setminus B^{\diamond})]^{\diamond}=A\cap(B\setminus B^{\diamond})^{\diamond}=A\cap \emptyset=\emptyset $.

Also, $(A\cap B)^{\diamond}\setminus (A\cap B^{\diamond})^{\diamond}\subseteq [(A\cap B)\setminus (A\cap B^{\diamond})]^{\diamond}= [A\cap  (B\setminus B^{\diamond})]^{\diamond}=\emptyset$ by Lemma ~\ref{10}. Hence, $(A\cap B)^{\diamond}\subseteq(A\cap B^{\diamond})^{\diamond}$ and,  we get $(A\cap B)^{\diamond}=(A\cap B^{\diamond})^{\diamond}$.  

Again $(A\cap B)^{\diamond}=(A\cap B^{\diamond})^{\diamond}\subseteq cl(A\cap B^{\diamond})$, since $\tau^{\diamond}$ is finer than $\tau$.
Due to $A\cap  B^{\diamond} \subseteq (A\cap B)^{\diamond}$, we have $cl(A\cap  B^{\diamond} )\subseteq cl((A\cap B)^{\diamond})=(A\cap B)^{\diamond}$.
Hence, $(A\cap B)^{\diamond}=cl(A\cap B^{\diamond})$.
\end{proof}


\begin{corollary}
Let $(X, \tau, \mathcal{P})$ be a primal topological space  such that $\tau$ is suitable for
$\mathcal{P}$. If $A\in \tau$ and $A^{c}\notin\mathcal{P}$, then $A\subseteq X\setminus X^{\diamond}$.
\end{corollary}

\begin{proof} Taking $B=X$ in Theorem ~\ref{16}, we get $(A\cap X)^{\diamond}=cl(A\cap X^{\diamond})$. Thus, $A^{\diamond}=cl(A\cap X^{\diamond})$, 
for all $A\in \tau$. Now if $A^{c}\notin\mathcal{P}$, then $A^{\diamond}=\emptyset$. Thus, $(A\cap X)^{\diamond}=cl(A\cap X^{\diamond})=\emptyset$. So, 
$A\cap X^{\diamond}=\emptyset$ by Theorem ~\ref{12}  and hence,  $A\subseteq X\setminus X^{\diamond}$.
\end{proof}
\section{Conclusion}

In this paper, we introduced $\Psi$ operator using primal \cite{aoi} and studied some fundamental properties. Moreover, we introduce topology suitable for a primal and proved some equivalent conditions. We hope that these results will find these suitable roles in topological research related to primal. \\ 

\bibliographystyle{amsplain,latexsym}

\end{document}